\documentclass{aptpub}
\usepackage{mathtools}
\usepackage{dsfont}
\usepackage{amsmath}
\usepackage{amssymb}

\authornames{Diego Marcondes and Cl\'audia Peixoto} 
\shorttitle{Random Coin Tossing} 

\begin{document}

\title{Random Coin Tossing with Unknown Bias} 

\authorone[Universidade de S\~ao Paulo]{Diego Marcondes} 
\authortwo[Universidade de S\~ao Paulo]{Cl\'audia Peixoto}

\addressone{Instituto de Matemática e Estat\'istica\\
Universidade de S\~ao Paulo\\
Rua do Mat\~ao, 1010, 05508-090, S\~ao Paulo, Brazil} 

\begin{abstract}
Consider a coin tossing experiment which consists of tossing one of two coins at a time, according to a renewal process. The first coin is fair and the second has probability $1/2 + \theta$, $\theta \in [-1/2,1/2]$, $\theta$ unknown but fixed, of head. The biased coin is tossed at the renewal times of the process, and the fair one at all the other times. The main question about this experiment is whether or not it is possible to determine $\theta$ almost surely as the number of tosses increases, given only the probabilities of the renewal process and the observed sequence of heads and tails. We will construct a confidence interval for $\theta$ and determine conditions on the process for its almost sure convergence. It will be shown that recurrence is in fact a necessary condition for the almost sure convergence of the interval, although the convergence still holds if the process is null recurrent but the expected number of renewals up to and including time $N$ is $O(N^{1/2+\alpha}), 0 \leq \alpha < 1/2$. It solves an open problem presented by Harris and Keane (1997). We also generalize this experiment for random variables on $L^{2}$ which are sampled according to a renewal process from either one of two distributions.
\end{abstract}



\section{Introduction}
Consider a sequence of heads and tails originated from a coin tossing experiment which consists of tossing one of two coins at a given time, according to a renewal process. The first coin is fair and the second has probability $1/2 + \theta$, $\theta \in [-1/2,1/2]$, $\theta$ unknown but fixed, of head. The biased coin is tossed at the renewal times of the process, and the fair one at all the other times. The main question about this experiment is whether or not it is possible to determine $\theta$ almost surely as the number of tosses increases, given only the probabilities of the renewal process and the observed sequence of heads and tails.

The work of \cite{RCT} treated a slightly different problem and left a question related to the above as an open problem. In their approach, an extension of a Random Walk on scenery model, the second coin is \textit{known} to be either fair or have a bias $\lambda$, i.e., $\mathbb{P}\{heads\} = 1/2(1+\lambda), \lambda \in [0,1]$. A method to determine if the second coin was fair or had a known bias $\lambda$ was developed based on the relationship between two convenient measures and conditions for the possibility of such determination were given. \cite{LEVIN} solved the main open problem of their paper and \cite{GEN1} and \cite{GEN2} generalized their problem for the case in which the bias of the second coin may differ from time to time. 

The case in which $\lambda$ is unknown was also treated by \cite{RCT}, but only for renewal times given by the return to the origin of a simple random walk on $\mathbb{Z}$. We will extend their method to general renewal processes, giving a sufficient condition for the almost sure determination of $\theta$, which is the same that allowed them to prove the possibility of such determination on the case in which $\lambda$ is unknown and the returns are given by a random walk on $\mathbb{Z}$.

This paper has three main objectives: to construct a confidence interval for $\theta$, to determine for which renewal processes it converges almost surely to $\theta$ and to generalize it for a class of random variables, as follows.

Suppose two random variables on $L^{2}$ with means $M$ and $M + \Delta$, respectively, in which $M$ is known and $\Delta \in \mathbb{R}$ is unknown, but fixed. The second random variable is sampled on renewal times of a renewal process, while the first one is sampled at all the other times. Again, one may ask if it is possible to determine $\Delta$ almost surely as the sample size increases, given only the sample of the random variables and the probabilities of the renewal process. On Section \ref{delta} we will construct a confidence interval for $\Delta$ similar to the one for $\theta$ and determine a sufficient condition on the renewal process for its almost sure convergence.

As an heuristic for what kind of renewal process it is possible to determine $\theta$ or $\Delta$ almost surely, we argue that the renewals must happen often for, otherwise, there will not be enough information about $\theta$ or $\Delta$ to differentiate it from random noise. Indeed, as more often the renewals, smaller will be the confidence intervals for $\theta$ and $\Delta$. In fact, recurrence is a necessary condition for the almost sure convergence of the interval. It will be shown that if the process is positive recurrent, then $\theta$ and $\Delta$ may be determined almost surely on the scenario described above. 

Nevertheless, if the process is null recurrent, but the expected number of renewals up to and including time $N$ is $O(N^{1/2+\alpha}), 0 \leq \alpha < 1/2$, it is also possible to determine $\theta$ and $\Delta$ almost surely, when $N$ diverges. We conclude then that a $O(N^{1/2+\alpha}), 0 \leq \alpha \leq 1/2$, expected number of renewals on the first $N$ times is a sufficient condition for the almost sure convergence of the confidence interval, for case $\alpha = 1/2$ is equivalent to the positive recurrence condition. This is the same condition used to prove Theorem 2 of \cite{RCT}. 

On Section \ref{notation} the problems above will be formally defined and some useful results will be presented. We then develop in Section \ref{theta} the confidence intervals for $\theta$ and $\Delta$ in a similar fashion, giving conditions for their almost sure convergence. On the last section we apply the obtained results to some renewal processes.

\section{Notation}
\label{notation}

Let $X_{n}, n = 1,\dots,N$, be independent random variables taking values in $A=\{0,1\}$ with probabilities $\mathbb{P}\{X_{n} = 1\} = \mathbb{P}\{X_{n} = 0\} = 1/2$. Analogously, let $X'_{n}, n =1,\dots,N$, be independent random variables taking values in $A$ with probabilities $\mathbb{P}\{X'_{n} = 1\} = 1- \mathbb{P}\{X'_{n} = 0\} = 1/2 +\theta, \theta \in [-1/2,1/2]$, $\theta$ unknown but fixed. We assume that the random variables $\{X_{1},\dots,X_{N},X'_{1},\dots,X'_{N}\}$ are independent. In the scheme of the Introduction, $\{X_{n}\}$ represent the tosses of the fair coin and $\{X'_{n}\}$ the tosses of the biased one.

Using a notation similar to \cite[Chapter~13]{Feller}, let $\xi$ be the event that represents a renewal and $\delta_{n} = \mathds{1}\{\xi \text{ happens at time n}\}, n \geq 1$, in which $\mathds{1}\{\cdot\}$ is the indicator function, and denote $\delta_{1}^{N} = \{\delta_{1},\dots,\delta_{N}\}$. Define $T = \min \{n\geq 1: \delta_{n} = 1\}$; denote $T = \infty$ if $\delta_{n} = 0, \forall n \geq 1$; and let $\mu = E(T)$ be the mean recurrence time of the process. We assume $\{\delta_{n}\}$ independent of $\{X_{n}\}$ and $\{X'_{n}\}$.

The probability distribution of the first renewal time is given by $\boldsymbol{f} = (f_{1},f_{2},\dots)$ in which $f_{n} = \mathbb{P}\{T = n\}, \forall n \geq 1$. Now, define the probability of having a renewal at time $n$ as $u_{n} = \mathbb{P}\{\delta_{n} = 1\}, \forall n \geq 1$, and denote $\boldsymbol{u} = (u_{1},u_{2},\dots)$. We then have the well-known recurrence relations given by
\begin{equation*}
\begin{cases}
u_{0} \coloneqq 1 \\
u_{n} = \sum\limits_{k=1}^{n} f_{k} u_{n-k}.
\end{cases}
\end{equation*}
Finally, denote the expected number of renewals up to and including time $N$ as $U_{N} \coloneqq \sum\limits_{n=1}^{N} u_{n}$. 

The renewal process associated with the experiment is completely defined by either the probability distribution $\boldsymbol{f}$ of $T$ or the probabilities $\boldsymbol{u}$ of having a renewal at time $n, \forall n\geq 1$. Therefore, it will be supposed that one of them is known.

The random variables that are actually sampled may be defined as
\begin{align*}
Y_{n} = \begin{cases}
X_{n}, & \text{if } \delta_{n} = 0 \\
X'_{n}, & \text{if } \delta_{n} = 1 \\
\end{cases}, & & n = 1, \dots, N.
\end{align*}
Note that $\{Y_{1}, \dots, Y_{N}\}$ are independent. Furthermore, the expectation of $\bar{Y}_{N} \coloneqq \sum_{n=1}^{N} Y_{n}/N$ is given by
\begin{equation}
E(\bar{Y}_{N}) = E(E(\bar{Y}_{N}|\delta_{1}^{N})) = \sum\limits_{n=1}^{N} \frac{1}{N} \Bigg[ \Big(\frac{1}{2}+ \theta\Big) u_{n} + \frac{1}{2} (1 - u_{n}) \Bigg] = \theta \frac{U_{N}}{N} + \frac{1}{2}.
\end{equation}

Using the notation above, the main problem of this paper may be stated as \textit{to determine $\theta$ almost surely when $N$ diverges, given only the sample $\{Y_{n}\}$ and either $\boldsymbol{f}$ or $\boldsymbol{u}$.}

A slightly different problem may also be treated in the same manner. Let $W_{n}, n = 1,\dots,N$, be independent and identically distributed random variables on $L^{2}$ with known mean $E(W_{1}) = M$ and variance $\sigma^{2}_{W} > 0$. Also, let $W'_{n}, n = 1,\dots,N$, be independent and identically distributed random variables on $L^{2}$ with mean $E(W'_{1}) = M + \Delta$, $\Delta \in \mathbb{R}$ unknown but fixed, and the same variance $\sigma_{W}^{2} > 0$. The variance $\sigma_{W}^{2}$ is supposed known.

Analogously, we define
\begin{align}
V_{n} = \begin{cases}
W_{n}, & \text{if } \delta_{n} = 0 \\
W'_{n}, & \text{if } \delta_{n} = 1 \\
\end{cases}, & & n = 1, \dots, N
\end{align}
and engage in determining $\Delta$ almost surely when $N$ diverges, given only $\{V_{n}\}$ and either $\boldsymbol{f}$ or $\boldsymbol{u}$. Note that the expectation of $\bar{V}_{N} \coloneqq \sum_{n=1}^{N} V_{n}/N$ is given by
\begin{equation}
E(\bar{V}_{N}) = E(E(\bar{V}_{N}|\delta_{1}^{N})) = \sum\limits_{n=1}^{N} \frac{1}{N} \Bigg[ (M + \Delta) u_{n} + M (1 - u_{n}) \Bigg] = \Delta \frac{U_{N}}{N} + M.
\end{equation}

The main result to be used in this paper is Hoeffding's Inequality for Bounded Random Variables, as stated below.

\textbf{Hoeffding's Inequality}: If $\{Y_{1},\dots, Y_{N}\}$ are independent random variables and $a_{n} \leq Y_{n} \leq b_{n}, n=1, \dots, N$, then for $\epsilon > 0$
\begin{equation}
\label{HOEG}
\mathbb{P}\Big\{|\bar{Y}_{N} - E(\bar{Y}_{N})| \geq \epsilon\Big\} \leq 2 \exp\Bigg\{\frac{-2N^{2}\epsilon^{2}}{\sum_{n=1}^{N}(b_{n} - a_{n})^{2}}\Bigg\}.
\end{equation}
Also, if $\{Y_{1},\dots, Y_{N}\}$ take values on $\{0,1\}$, (\ref{HOEG}) reduces to
\begin{equation}
\label{HOEB}
\mathbb{P}\Big\{|\bar{Y}_{N} - E(\bar{Y}_{N})| \geq \epsilon\Big\} \leq 2 \exp\{-2N\epsilon^{2}\}.
\end{equation}
A proof for (\ref{HOEG}) is given in \cite{HOE}.

Note that the bound in (\ref{HOEG}) does not depend on the variance or mean of $Y_{n}$. Analogously, the bound on (\ref{HOEB}) does not depend on $\mathbb{P}\{Y_{n} = 1\}$ that may also be unknown. Therefore, those bounds are quite useful in practice, for those quantities are often unknown. In fact, the bound in (\ref{HOEB}) is the best that can be obtained under these conditions.

\section{$\theta$ almost surely}
\label{theta}

In this section we will deduce a confidence interval for $\theta$ and give conditions for its almost sure convergence. First, suppose that it is known that, for $\epsilon > 0$ and $0 < \gamma < 1$,
\begin{equation}
\mathbb{P}\{|\bar{Y}_{N} - E(\bar{Y}_{N})| < \epsilon\} \geq \gamma.
\label{DIF}
\end{equation}
Therefore, as
\begin{align*}
-\epsilon < \bar{Y}_{N} - E(\bar{Y}_{N}) < \epsilon \iff & \frac{-N\epsilon + N \bar{Y}_{N} - N/2 }{U_{N}} < \theta < \frac{N\epsilon + N \bar{Y}_{N} - N/2 }{U_{N}}
\end{align*}
an interval of confidence $\gamma$ for $\theta$ is given by
\begin{equation*}
CI(\theta;\gamma) = \Bigg[\frac{N}{U_{N}}\Big(-\epsilon + \bar{Y}_{N} - \frac{1}{2}\Big);\frac{N}{U_{N}}\Big(\epsilon + \bar{Y}_{N} - \frac{1}{2}\Big)\Bigg] \coloneqq \Big[L_{I}^{N};L_{S}^{N}\Big]
\end{equation*}
in which $\epsilon$ is a function of $\gamma$.

The next step is to determine $\epsilon$ given a confidence $\gamma$, i.e., write $\epsilon$ explicitly as a function of $\gamma$. A first idea is to apply the Chebyshev's Inequality and get that
\begin{equation*}
\mathbb{P}\{|\bar{Y}_{N} - E(\bar{Y}_{N})| < \epsilon\} \geq 1 - \frac{Var(\bar{Y}_{N})}{\epsilon^{2}} \geq 1 - \frac{1}{4N\epsilon^{2}} = \gamma \! \! \! \implies \! \! \! \epsilon_{c} \coloneqq (4N(1-\gamma))^{-1/2}.
\end{equation*}
As $Var(\bar{Y}_{N})$ depends on $\theta$, which is unknown, the $1/4N$ bound may be used, although it might not be a tight bound, especially for values of $\theta$ far from zero. 

A better bound for the probability (\ref{DIF}) is given by Hoeffding's Inequality (\ref{HOEB}) and, then,
\begin{equation*}
\mathbb{P}\{|\bar{Y}_{N} - E(\bar{Y}_{N})| < \epsilon\} \geq 1 - 2e^{-2N\epsilon^{2}} = \gamma \implies \epsilon_{h} \coloneqq \Bigg(2N\ln\Big(\frac{1-\gamma}{2}\Big)\Bigg)^{-1/2}.
\end{equation*}
The confidence interval given by $\epsilon_{h}$ is smaller than the one given by $\epsilon_{c}$. In fact, the interval given by $\epsilon_{h}$ is the smallest one in this scenario, i.e., when $\theta$ is unknown and no asymptotic distributions are assumed. Therefore, the best $\epsilon$ we are able to get is $O(N^{-1/2})$.

A third interval is given by the asymptotic normal approximation of the distribution of $\bar{Y}_{N}$ given by the Lindberg-Feller Central Limit Theorem \cite[page~128]{rao}, applying again the $1/4N$ bound for its variance, as
\begin{equation*}
\mathbb{P}\Bigg\{\Big|\bar{Y}_{N} - E(\bar{Y}_{N})\Big| < \Phi^{-1}\Big(\frac{1-\gamma}{2}\Big)\sqrt{\frac{1}{4N}}\Bigg\} \geq \gamma \implies \epsilon_{a} = \Phi^{-1}\Big(\frac{1-\gamma}{2}\Big)\sqrt{\frac{1}{4N}}
\end{equation*}
for great values of $N$, in which $\Phi(\cdot)$ is the cumulative distribution function of the standard normal distribution. The interval given by $\epsilon_{a}$ is better than the above, although it may used only for great values of $N$. When $N$ diverges, the intervals given by $\epsilon_{c}, \epsilon_{h}$ and $\epsilon_{a}$ are equivalent. Many other bounds may be applied to (\ref{DIF}) in order to get $\epsilon$, e.g., see the ones given by \cite{Bennet}, \cite{Bernstein} and \cite{Ngo}. 

We now engage in finding conditions for the almost sure convergence of the confidence interval to $\theta$. The theorem below gives a sufficient condition on $U_{N}$ for the almost sure convergence of the interval to $\theta + k$, in which $k$ is a known constant, and, consequently, the almost sure determination of $\theta$.

\begin{theorem}
	\label{theorem}
	Let k be a known constant. If $U_{N} = O(N^{1/2+\alpha}), 0 \leq \alpha \leq 1/2$, then $\lim\limits_{N \rightarrow \infty} L_{S}^{N} = \lim\limits_{N \rightarrow \infty} L_{I}^{N} = \theta + k$ almost surely.
\end{theorem}

\begin{proof}
	By the Kolmogorov's Strong Law of Large Numbers \cite[page~114]{rao}, we have that
	\begin{equation*}
	\bar{Y}_{N} \xrightarrow[N \rightarrow \infty]{a.s} \lim\limits_{N \rightarrow \infty} E(\bar{Y}_{N}) = \lim\limits_{N \rightarrow \infty} \theta \frac{U_{N}}{N} + \frac{1}{2}.
	\end{equation*}
	Therefore, the almost sure limit of the interval is given by
	\begin{align*}
	\lim\limits_{N \rightarrow \infty} L_{S}^{N} = \lim\limits_{N \rightarrow \infty} L_{I}^{N} = & \lim\limits_{N \rightarrow \infty} \frac{N}{U_{N}} \Bigg[E(\bar{Y}_{N}) - \frac{1}{2} + O(N^{-1/2})\Bigg] \\ = & \lim\limits_{N \rightarrow \infty} \frac{N}{U_{N}} \Bigg[\theta \frac{U_{N}}{N} + O(N^{-1/2})\Bigg] = \theta + \lim\limits_{N \rightarrow \infty} \frac{O(N^{1/2})}{U_{N}}
	\end{align*}
	in which the $O(N^{-1/2})$ element is one of the $\epsilon$ presented above, for a fixed confidence $\gamma$. Thus, denoting $\lim\limits_{N \rightarrow \infty} \frac{O(N^{1/2})}{U_{N}} = k$, the result follows.
\end{proof}

Note that $k = \lim\limits_{N \rightarrow \infty} \frac{O(N^{1/2})}{U_{N}}$ is known, for \textbf{u} is known. Therefore, subtracting $k$ from each side of the confidence interval, we may determine $\theta$ almost surely when $N$ diverges, as desired. In fact, $k = 0$ if $U_{N} = O(N^{1/2+\alpha}), 0 < \alpha \leq 1/2$, and is positive, yet known, if $U_{N} = O(N^{1/2})$.

On the one hand, if the process is positive recurrent then $\lim\limits_{N \rightarrow \infty} \frac{N}{U_{N}} = \mu$ which implies that $U_{N} = O(N)$ and the convergence holds. On the other hand, if the process is null recurrent then $\lim\limits_{N \rightarrow \infty} \frac{N}{U_{N}} = \infty$, which implies that $U_{N}  = O(N^{1/2+\alpha})$, for some $-1/2 < \alpha < 1/2$. Hence, the convergence holds for null recurrent processes, but only if $0 \leq \alpha < 1/2$. If the process is transient, then $\lim\limits_{N \rightarrow \infty} U_{N} < \infty$ \cite[Theorem~1(a)]{Feller2} and the interval does not converge. It proves that recurrence is a necessary condition for the convergence of the interval.

The convergence of the interval to $\theta$ for null recurrent processes is quite interesting, for one's intuition could wrongfully infer that $\mu < \infty$ is a necessary condition for the convergence, as it is expected that the renewals must happen often enough for the convergence to hold.

\section{Extension to random variables on $L^{2}$}
\label{delta}

The confidence interval established above may be immediately extended to random variables on $L^{2}$ in the following way. Suppose that, for $\epsilon > 0$ and $0 < \gamma < 1$,
\begin{equation}
\mathbb{P}\{|\bar{V}_{N} - E(\bar{V}_{N})| < \epsilon\} \geq \gamma.
\label{DIFDELTA}
\end{equation}
Therefore, as
\begin{align*}
-\epsilon < \bar{V}_{N} - E(\bar{V}_{N}) < \epsilon \iff & \frac{-N\epsilon + N \bar{V}_{N} - NM }{U_{N}} < \Delta < \frac{N\epsilon + N \bar{V}_{N} - NM }{U_{N}}
\end{align*}
and an interval of confidence $\gamma$ for $\Delta$ is given by
\begin{equation*}
CI(\Delta;\gamma) = \Bigg[\frac{N}{U_{N}}\Big(-\epsilon + \bar{V}_{N} - M\Big);\frac{N}{U_{N}}\Big(\epsilon + \bar{Y}_{N} - M\Big)\Bigg] \coloneqq \Big[L_{I}^{N};L_{S}^{N}\Big] 
\end{equation*}
in which $\epsilon$ is function of $\gamma$.

Again, it is necessary to write $\epsilon$ in function of $\gamma$. On the one hand, if the random variables $\{V_{n}\}$ are bounded, we may apply the Hoeffding's Inequality (\ref{HOEG}) to get $\epsilon$. On the other hand, if they are not bounded, we may apply the Chebysev's Inequality or the asymptotic normal approximation given by the Lindberg-Feller Central Limit Theorem to get $\epsilon$ in function of the variance of $W_{t}$ and $W'_{t}$, which is known. Either way, $\epsilon$ will be $O(N^{1/2})$ and we have the corollary below, which the proof is similar to that of Theorem \ref{theorem}.

\begin{corollary}
	Let k be a known constant. If $\mu = O(N^{1/2+\alpha}), 0 \leq \alpha \leq 1$, then $\lim\limits_{N \rightarrow \infty} L_{S}^{N} = \lim\limits_{N \rightarrow \infty} L_{I}^{N} = \Delta + k$ almost surely.
	\label{cor}
\end{corollary}

\section{Applications}
\label{apliccations}

If the renewals are given by the returns to the origin of a simple random walk on $\mathbb{Z}$, then $U_{N} = O(N^{1/2})$ and, by Theorem \ref{theorem}, it is possible to determine $\theta$ almost surely as $N$ diverges. Note that it is in accordance with \cite[Theorem~2]{RCT}. However, if the renewal times are given by the returns to the origin of a simple random walk on $\mathbb{Z}^{2}$, then the condition of Theorem \ref{theorem} is not satisfied and $\theta$ cannot be determined applying this method. If the renewals are given by simple random walks on $\mathbb{Z}^{d}, d \geq 3$, again it is not possible to determine $\theta$ almost surely on our approach, for the process is transient.

As commented by \cite{RCT}, the condition on the process which allows the almost sure determination of $\theta$ when it is known, i.e., the condition $\sum_{n=1}^{\infty} u_{n}^{2} = \infty$, does not imply nor is implied by the condition of Theorem \ref{theorem}. Indeed, we are treating a different problem from that of \cite{RCT}, which is closer related to the theory of statistics than to measure theory, for we are interested in estimating an unknown parameter when the simple size diverges.

\section{Final Remarks}

Our approach does not prove that the condition of Theorem \ref{theorem} is a necessary one for the almost sure determination of $\theta$, however it is expected that there must be a critical point $\alpha$ such that $U_{N}$ must be at least $O(N^{1/2 + \alpha})$ for the convergence to hold. This point might actually be $\alpha = 0$, as the $N^{1/2}$ convergence rate appears intrinsically in a lot of scenarios all over probability theory. As a matter of fact, the $O(N^{1/2})$ of $U_{N}$ of the process given by the random walk on $\mathbb{Z}$ is what allowed \cite{RCT} to prove its Theorem 2. 

The method above may also be seen from a statistical point of view, in which one is interested in estimating $\theta$ or $\Delta$. The model on $\Delta$ might be of use for statistical quality control, for the random variables presented may be seen as the behaviour of a machine that, when working properly, performs a job that is measured by a random variable with mean $M$. Although, it is known that it may not work properly from time to time, according to a renewal process, and when it does not work properly, it performs its job according to a different random variable with mean $M + \Delta$. In most applications, those random variables may be considered bounded and, if not, are hardly of infinite variance.

The model on $\theta$ and $\Delta$ may be also extended for the case in which $\theta$ and $\Delta$ differ from time to time or the probability of head of the first coin is $p \in (0,1)$. This is an interesting topic for further research, along with the establishment of a necessary condition for the possibility of determining $\theta$ almost surely when it is unknown.

\acks
We would like to thank A. Simonis who presented the problem to us and L. G. Esteves for his careful reading of the manuscripts.

%
%
%
%

\end{document}